\documentclass[11pt]{amsart}
\usepackage{times}
\usepackage[T1]{fontenc}
\usepackage{amssymb, amsthm, amsmath}
\usepackage{mathrsfs}
\usepackage[all]{xy}

\usepackage[active]{srcltx}

\usepackage{todonotes}

\usepackage{setspace}

\usepackage{hyperref}

\DeclareMathOperator{\acl}{acl}

\newtheorem{introtheorem}{Theorem}
\newtheorem*{thm*}{Theorem}

\newtheorem*{gen-dif}{\fbox{{\large A}} \hypertarget{Agen-dif}{Gen-Dif}}

\newtheorem*{min-balln}{\fbox{{\large A}} \hypertarget{Amin-ball}{Cballs}}

\theoremstyle{definition}

\newcommand{\Nn}{{\mathbb{N}}}

\newcommand{\m}{\textbf{m}}

\newcommand{\CK}{{\mathcal K}}

\newcommand{\CM}{{\mathcal M}}

\newcommand{\CC}{{\mathbb C}}

\newcommand{\CG}{{\mathcal G}}
\newcommand{\CS}{\mathcal S}

\newcommand{\0}{\emptyset}

\renewcommand{\phi}{\varphi}

\def\bm{\mathfrak m}

\def\sub{\subseteq}

\newcommand{\from}{\leftarrow}
\newcommand{\xto}{\xrightarrow}

\renewcommand{\m}[1]{\mathrm{#1}}

\newcommand{\new}{\newcommand}

\new{\spar}
[1]{\medskip\refstepcounter{subsection}\noindent\arabic{section}.\arabic{subsection}.\label{#1}}

\new{\sspar}
[1]{\medskip\refstepcounter{subsection}\noindent\arabic{section}.\arabic{subsection}.\arabic{subsubsection}.\label{#1}}

\theoremstyle{definition}
\newtheorem{sRemark}[subsection]{Remark}
\newtheorem{sCorollary}[subsection]{Corollary}
\newtheorem{sLemma}[subsection]{Lemma}
\newtheorem{sProposition}[subsection]{Proposition}
\newtheorem{sTheorem}[subsection]{Theorem}
\newtheorem{sClaim}[subsection]{Claim}
\newtheorem{sDefinition}[subsection]{Definition}
\newtheorem{sFact}[subsection]{Fact}

\newenvironment{claimproof}[1][\proofname]
{%
\proof[#1]%
}
{%
\endproof%
}

\newcommand{\inv}{^{-1}}
\new{\kbar}{\overline k}
\new{\ord}{\opnm{ord}}
\new{\mM}{\mathfrak{m}}
\new{\nN}{\mathfrak n}
\new{\Supp}{\opnm{Supp}}

\newcommand{\Spec}{\operatorname{Spec}}

\newcommand{\Div}{\operatorname{Div}}

\newcommand{\Aut}{\mathrm{Aut}\,}

\newcommand{\si}{\sigma}

\newcommand{\ga}{\gamma}
\newcommand{\al}{\alpha}
\newcommand{\be}{\beta}

\new{\mMp}{{\mM'}}

\newcommand{\FF}{\mathbb{F}}

\new{\opnm}{\operatorname}
\new{\Fr}{\opnm{Fr}}

\new{\Fq}{F_q}
\new{\Fqbar}{{\overline F_q}}

\new{\Gal}{\opnm{Gal}}

\new{\ZZ}{\mathbb{Z}}
\new{\Fbar}{{\overline F}}
\new{\PP}{\mathbb{P}}
\new{\Gm}{\mathbb{G}_m}

\title{A curve and its abstract \emph{generalized} Jacobian}
\author{Benjamin Castle}
\address{BC: Department of mathematics, University of Illinos at Urbana Champaign}
\email{btcastl2@illinois.edu}
\thanks{BC was supported by NSF grant DMS-2452735.}

\author{Ishai Dan-Cohen}
\address{ID: Department of Mathematics, Ben-Gurion University of the Negev.}
\email{ishaida@bgu.ac.il}
\thanks{ID was supported by ISF grant number 621/21. Political disclaimer: ID asks not to be considered responsible for the actions of any government that does not fully embrace the principles of democracy.}

\author{Assaf Hasson}
\address{AH: Department of Mathematics, Ben-Gurion University of the Negev.}
\email{hassonas@math.bgu.ac.il}
\thanks{AH was supported by ISF grant number 555/21}

\date{\today}
\begin{document}

\thanks{}

\begin{abstract}
To a smooth proper curve $C$ over a field $k$ equipped with a $k$-point $c$ and an effective divisor $\mM$ coprime to $c$, one may associate the abstract group $J_\mM(\bar k)$ of $\overline k$-points of the generalized Jacobian, as well as a subset
\[
\tag{*}
\big(C\setminus \Supp(\mM)\big)(\bar k) \subset J_\mM(\bar k).
\]
We show that the data $(C,c,\mM)$ can be retrieved from (*) up to a twist by an automorphism of $\overline k$, proving a conjecture of Booher and Voloch from \cite{BooVol}. By a result of Booher and Voloch this shows that when $k$ is a finite field, the same data may also be retrieved from $L$-functions of characters of certain Galois extensions of the function field of $C$.
The proof is a generalization of Zilber's well known work "A curve and its abstract Jacobian".
\end{abstract}

\maketitle

\section{Introduction}

In \cite{ZilJac} Zilber proves that, given an algebraically closed field $K$,  a smooth irreducible projective algebraic curve $C$ over $K$ of genus at least 2 can be recovered from the abstract abelian group $(J_C(K), +)$ of $K$-rational points of the Jacobian $J_C$ of $C$, augmented by a unary predicate for the image of $C$ in $J_C$ under an Abel-Jacobi map. In \cite[Theorem 2.4]{BoVo1} Booher and Voloch give a geometrically more precise statement of Zilber's theorem, for the case where $K=\overline {\mathbb F}_q$, from which they deduce (Theoren 3.1, \emph{loc. cit.}) a new proof of a theorem of Mochizuki and Tamagawa on the reconstruction of a smooth projective curve $C$ over a finite field from $\pi_1(C)$. In \cite[Conjecture 2.6]{BooVol} Booher and Voloch conjecture that Zilber's theorem can be extended to reconstruct a smooth projective irreducible curve and a modulus $\bm$ on $C$ from the pure abelian group of $K$-rational points of $J_\bm(C)$ -- the \emph{generalized} Jacobian\footnote{Readers unfamiliar with generalized Jacobians may consult Serre's book \cite{SerreGJ} for all the necessary background. See also the proof of Theorem \ref{AF12} for the definition.} of $(C,\bm)$ -- expanded by a predicate for the image of $C$ in $J_\bm(C)$ under an Abel-Jacobi map. 



The current work aims to prove this conjecture. Our main theorem in outline is as follows.

\begin{introtheorem}\label{T: main}
Let $C$, $C'$ be smooth proper geometrically connected pointed curves over a field $k$ and $\mM$, $\mM'$ effective divisors coprime to the base points. Let $U$, $U'$ be the complements of the supports of $\mM$, $\mM'$ in $C$, $C'$. Assume the associated generalized Jacobians $J_\mM$, $J_{\mM'}$ have dimension at least $2$. Suppose
\[
\psi: J_\mM(\overline k) \to J_{\mM'}(\overline k)
\]
is an isomorphism of abstract groups of points with values in an algebraic closure $\overline k$ of $k$.
\begin{enumerate}
\item
Assume 
\[
\psi \big( U(\overline k) \big)
=
U'(\overline k). \footnote{To keep notation simpler, we will not distinguish between $U$ and its image in $J_\mM$ under the Abel-Jacobi map.}
\]
Then $\psi$ is induced by an isomorphism over $\overline k$ of $(C', \mM')$ with $(C, \mM)$ up to a twist by an automorphism of $\overline k$.
\item
Let $k = \FF_q$ and assume
\[
\psi \big( U(\FF_{q^m}) \big)
=
U'(\FF_{q^m})
\]
for all $m$. Then $\psi$ is induced by an isomorphism over $\FF_q$ of $(C', \mM')$ with $(C, \mM)$ up to a twist by Frobenius. 
\end{enumerate}

\end{introtheorem}

See Theorem \ref{BC33} and Proposition \ref{P: ACF version of main thm} for the precise statements. As Booher and Voloch \cite{BoVo1} explain, their conjecture implies that a pointed curve $(C,[M])$ and an effective divisor $\mM$ coprime to the base point may be retrieved from L-functions of certain Galois extensions of the function field.


Actually, Conjecture 2.6 of \cite{BoVo1} differs from our theorem in that Booher and Voloch assume only that $\psi$ induces a bijection between $\overline \FF_q$-points of the respective open curves $U$, $U'$. This is insufficient, as can already be seen in case $\mM = 0$ by considering any two nonisomorphic curves of genus $\ge 2$ over a finite field that become isomorphic over a finite extension. Despite this, their application involving L-functions is unharmed for the simple reason that at this point they \textit{do} impose a stronger condition involving $\FF_{q^m}$-points for all $m$. We include an account of this application in the essentially expository Section 4, culminating in our Corollary \ref{BA36}, which is stated as Conjecture 2.4 of \cite{BoVo1}.

The proof of our main result consists of a model-theoretic part (Section 2) and a geometric part (Section 3). The model-theoretic part of the proof follows Zilber's argument, and like Zilber's theorem, it works over any algebraically closed field. In fact, for any algebraically closed $K$, $K'$ an isomorphism $\psi: J_{\bm}(K)\to J_{\bm'}(K')$ mapping $U(K)$ onto $U'(K')$ implies that $K\cong K'$ and that $\psi$ is induced by the composition of an isomorphism of fields with a bijective isogeny. The geometric part of the proof then recovers $(C, \bm)$ from the data, as well as the precise form of the isomorphism, when the data is defined over a finite field and $K$ is its algebraic closure.

\bigskip

\subsection*{Acknowledgements}
We wish to thank Felipe Voloch for his encouragement and for several helpful email exchanges. We thank Will Sawin for his responses to questions related to this work on \textit{mathoverflow}. Finally, we are particularly indebted to Jeremy Booher: in an earlier draft, the isomorphism of curves constructed in Theorem \ref{BC33} was defined only after a finite field extension, and this caveat infected also our Corollary \ref{BA36}; the strategy for descending down to $\FF_q$ was suggested to us by Booher.

\section{Model theory}

Zilber's main result of \cite{ZilJac}, described informally in the opening paragraph of this paper, is essentially Theorem \ref{T: main} applied to the case $\mM=(0)$ (at least up to a reformulation due to Booher and Voloch \cite[Theorem 2.4]{BooVol}). The model-theoretic part of our argument follows Zilber's original proof closely, with some necessary generalizations (and some simplifications). The model-theoretic content is summarized by Theorem \ref{T: Zil} below; it roughly states that a commutative algebraic group $G$, generated by a (not necessarily closed) algebraic curve $C$, can be reconstructed (up to a quotient by a finite subgroup) from the abstract group of points with values in a fixed algebraically closed field, augmented by a unary predicate for $C$. The proof strategy is described in detail in \cite{CasHasAV}, so we will be brief. 

\spar{}
The main ingredient in Zilber's proof is a special case, \cite{Ra}, of "Zilber's restricted trichotomy" (referred to in Zilber's paper as "Rabinovich's Theorem"). 
In \cite{CaHaYe}, Zilber's Restricted Trichotomy Conjecture was finally fully proved, allowing some of the above-mentioned generalizations and simplifications (arising, in part, by avoiding the need to meet the additional assumptions of Rabinovich's theorem).  Roughly, Zilber's Restricted Trichotomy asserts that if $K$ is an algebraically closed field, and $S$ is a constructible subset of $K^n$ (in the usual algbro-geometric sense), then an isomorphic copy $F$ of $K$ can be recovered from any small fraction of the geometric structure on $S$, provided a necessary non-triviality assumption is fulfilled\footnote{The failure of this non-triviality assumption, referred to below as 1-basedness, splits into two subcases: a totally disintegrated case, with no possible algebraic structure, and a "module like"  case. Whence, "Zilber's Trichotomy". This is a restriction to algebraically closed fields of Zilber's  Trichotomy Conjecture, that in full generality, was disproved by Hrushovski \cite{Hr1}.}. Although Zilber's restricted trichotomy is only used in the proof as a black box, let us state it for the sake of clarity and completeness of the exposition, explaining the geometric content of the model-theoretic terminology. The following is a corollary to Theorem 1 of \cite{CaHaYe}:
\begin{sFact}[Zilber's Restricted Trichotomy]
    Let $K$ be an algebraically closed field, $S$ a constructible subset of $K^n$. Consider a structure $\mathcal S$ with universe $S(K)$ and atomic relations that are all $K$-constructible subsets of cartesian powers of $S(K)$. If $\CS$ is not 1-based then an algebraically closed field $F$, $K$-definably isomorphic to $K$, is interpretable in $\CS$. 
\end{sFact}

Let us explain, very briefly, the model theoretic terminology used above. Readers unfamiliar with model-theoretic language may wish to further consult the appendix of \cite{CasHasAV} for a more detailed geometric interpretation of these notions. For completeness and greater readability, additional terminology will be explained as we proceed.  We recall that a \textit{structure $\CS$} is a non-empty set $S$ (called the universe of the structure) equipped with a collection of subsets of cartesian powers of $S$ (called the \textit{atomic} or basic relations of $\CS$)\footnote{The collection of symbols used to designate the atomic relations is the \textit{language} of the structure.}. In our setting, the universe of the structure will be the set of $K$-rational points of some algebraic group (the Jacobian of a curve in Zilber's theorem, the generalized Jacobian $J_\mM$ in the present paper). The atomic relations are the group operation (rather, its graph) and a unary predicate for the image of the curve $C$ under the Abel-Jacobi map. 

Given a structure $\CS$, a $\0$-definable set is a subset of $S^n$ (any $n$) whose elements are those satisfying a formula in the language consisting precisely of the atomic relations of $\CS$. 
E.g., in our setting, $\phi(x):=\exists z (z+z=x)$ defines those elements of $J_\mM$ divisible by $2$. For a parameter set  $A\subseteq S$ a set is $A$-definable if it is $\0$-definable in the structure $\CS$ augmented by constants for each element of $A$. E.g., if $\CS$ is a group and $g$ is any element then the set $\phi(x):=\exists z (x=zgz^{-1})$ is the set of conjugates of $g$ in $S$, and it is $g$-definable. By a set definable in $\CS$ (or an $\CS$-definable set) we mean an  $S$-definable set in $\CS$. A definable function is one whose graph is a definable set. 

It is an easy exercise to verify by induction that the $\0$-definable sets in a structure $\CS$ are precisely those obtained by closing the atomic sets, together with all possible diagonals, under finite cartesian products, finite Boolean combinations and coordinate projections. The $A$-definable sets are then those obtained as fibres over points all of whose coordinates are in $A$ of coordinate projections of $\0$-definable sets.

Given a structure $\CS$, an $\CS$-definable group is a definable set $G\sub S^n$ (some $n$), equipped with an $\CS$-definable function $\cdot: G^2\to G$ turning $G$ into a group. If $\cdot$ is a group operation on $S$, we say that $\CS$ expands a group. Definable fields are defined analogously. In order not to overload the notation, we will not distinguish notationally between the field $K$, and the field structure $\CK:=(K,+,\cdot)$ with universe $K$. We remark that, by a theorem of Chevalley, over an algebraically closed field $K$, the $K$-definable subsets of $K^n$ are precisely the constructible subsets of affine $n$-space in the sense of algebraic geometry. 

For expansions of groups, the non-triviality assumption in the statement of Zilber's Restricted Trichotomy can be formulated quite simply. In general, the notion of 1-based structures is an abstraction of the idea of linearity. In our setting, if $\CS$ is an expansion of a group, 1-basedness is equivalent to the condition that for all $n$, all definable subsets of $S^n$ are finite boolean combinations of cosets of definable subgroups of $S^n$ (see the main result of \cite{HrPi87} for details). In particular, this condition implies that every $S$-definable irreducible curve in $S^n$ agrees with a coset of a subgroup up to finitely many points.

Thus, the structure $\mathcal J:=(J_\mM(K), +, U(K))$ is not 1-based because $U(K)$ is not (even up to finitely many points) a coset\footnote{It is here that the assumption that $\dim(J_\mM(K))\geq 2$ is used. E.g., if $C$ is an elliptic curve and $m=0$ then $J_\mM=C$ and the structure $\mathcal J$ is just a pure abelian group, which is 1-based.} of a subgroup of $J_\mM$. So Zilber's Restricted Trichotomy applies. Its conclusion means that a field $F$ is interpretable in the structure $\mathcal J$. Let us note that for this exposition, the distinction between interpretable\footnote{An interpretable set is one obtained as the quotient of a definable set by a definable equivalence relation.} and definable sets is immaterial and is used solely for accuracy. It may be worth mentioning that while the proof of Zilber's Restricted Trichotomy is, in the most part constructive, it does not provide, at that level of generality, a formula defining a field.   \\

\spar{}
We now proceed to describe the conclusion of the argument. To recap, Zilber's Restricted Trichotomy allows us to construct in the structure $\mathcal J:=(J_\mM(K),+, U(K))$ a field $F$ that is $K$-definably isomorphic to $K$. Once a field $F$ is constructed definably in $\mathcal J$, the conclusions of Zilber's theorem of \cite{ZilJac} and our version thereof (Theorem \ref{T: Zil} below) are quite similar. An easy argument (see, e.g., \cite[Lemma 2.3]{CasHasAV}) shows that to conclude it is enough to show that our constructible set $J_\mM(K)$ can be endowed with the structure of an algebraic variety over the field $F$. By \cite[Proposition 4.12]{CaHa}, towards this end, it will suffice to construct a $\mathcal J$-definable injection of $J_\mM(K)$ into $F^n$ (some $n$). 
The construction of such an injection is standard, but slightly delicate. It constitutes the main bulk of the current section and builds on techniques introduced in \cite{CasHas} and in \cite{CasHasAV}.

To state what is achieved in this part of the argument in a more concise way, we recall the meaning of the phrase `$V$ geometrically determines $G$' from \cite{CasHasAV}. In that paper, $G$ was an algebraic group and $V$ was an irreducible closed subset. In fact, the definition and associated technology work if $V$ is \textit{any} constructible subset (and this is useful to us, because the set $U\subset J_m$ is not closed, and a priori we do not know the map $\phi$ respects closure). Thus, we will adopt this added generality:

\begin{sDefinition}
\label{D: determines}
Let $G$ be an algebraic group over an algebraically closed field $K$, and let $X$ be a constructible subset of $G$. We say that $X$ \textit{geometrically determines $G$} if the following holds: let $G'$ be another algebraic group over an algebraically closed field $K'$, with a constructible subset $X'$. Let $\phi:G(K)\rightarrow G'(K')$ be a group isomorphism satisfying $\phi(X(K))=X'(K')$. Then $\phi$ decomposes as  $\phi=f\circ\sigma$, where $\sigma$ is induced by a field isomorphism $K\rightarrow K'$ and $f$ is a bijective isogeny of algebraic groups over $K'$.
\end{sDefinition}

In this terminology, Zilber's theorem states that a smooth algebraic curve (irreducible, projective, of genus at least 2) geometrically determines its Jacobian. As an aside, let us note that if $G$ is a simple algebraic group then, by a model theoretic version of the Borel-Tits theorem due to Poizat, \cite{PoiFields}, the empty set geometrically determines $G$.  

It is convenient to restate Zilber's strategy (as outlined above) using more precise terminology. Given an algebraically closed field $K$, a $K$-\textit{relic} is a model theoretic structure all of whose $\0$-definable sets (and, in particular, its universe) are $K$-constructible. A $K$-relic $\CM$ with universe $M$ is \emph{full} if every $K$-constructible subset of $M^n$ (any $n$) is $\CM$-definable. The main part of Zilber's proof then amounts to showing that the $K$-relic $\mathcal J_C:=(J(K), +,C(K))$ 
is full. By \cite[Lemma 2.3]{CasHasAV}, this implies the main result. 

Let $(C,[M])$ be a pointed, smooth, irreducible, projective algebraic curve over an algebraically closed field $K$ and $\bm$ a modulus on $C$ (disjoint from $M$). We consider the relic $\mathcal J_\bm:=(J_\bm(K), +, U(K))$ where $U$ is the complement of the support of $\bm$ in $C$. 
The model-theoretic part of our argument consists of proving: 

\begin{sTheorem}
\label{T: Zil}
The relic $\mathcal J_m$ is full, and thus $U$ geometrically determines $J_\bm$.   
\end{sTheorem}

Note that the implication between the two clauses above requires a version of \cite[Lemma 2.3]{CasHasAV} allowing the non-closed case $X=U$. In fact, the original statement adapts directly. Let us state it in slightly more general language:

\begin{sLemma}
\label{L: full implies determines}
Let $K,K'$ be algebraically closed fields. Let $\mathcal M$ be a full $K$-relic, and $\mathcal M'$ a full $K'$-relic in the same language as $\mathcal M$. Let $\phi:\mathcal M\rightarrow\mathcal M'$ be an isomorphism of relics\footnote{I.e., an isomorphism of first order structures: a bijection mapping each atomic relation in one structure bijectively onto the atomic relation designated by the same symbol in the image structure. In the notation of our main result, the isomorphism $\psi$ is such.}.
\begin{enumerate}
\item $\phi$ decomposes as $\phi=f\circ\sigma$, where $\sigma$ is induced by a field isomorphism $K\rightarrow K'$ and $f$ is a $K'$-definable isomorphism of $K'$-relics.
\item If $\mathcal M=(G(K),\cdot,...)$ and $\mathcal M'=(G'(K'),\cdot,...)$ are expansions of algebraic groups\footnote{In our setting, by a theorem of Weil-Hrushovski, all definable groups are algebaric.}, then we can choose $f$ and $\sigma$ so that $f$ is a bijective isogeny of algebraic groups over $K'$.
\item In particular, if $G$ is any algebraic group over $K$, $X\subset G$ is constructible, and $\mathcal G=(G(K),\cdot,X(K))$ is full, then $X$ geometrically determines $G$. 
\end{enumerate}
\end{sLemma}

\begin{proof} By following the proof of \cite[Lemma 2.3]{CasHasAV} verbatim. Namely, the proof of that lemma up to the last paragraph gives precisely the decomposition (1). Then as in the last paragraph of that proof, quantifier elimination for the theory of algebraically closed fields lets us write any definable isomorphism of algebraic groups over $K'$ as an inverse Frobenius power composed with an isogeny -- and the inverse Frobenius power can be absorbed to the original field isomorphism 
$K\rightarrow K'$. This gets (2), and (3) is a restatement of (2) for the particular relic $\mathcal G$.
\end{proof}

\spar{} Thus, to prove Theorem \ref{T: Zil}, we just need to show that $\mathcal J_\mM$ is full. 
The proof will involve the notion of \textit{internality}: roughly speaking, the Zilber trichotomy will provide us with a field $F$, and an internality argument (i.e. internality\footnote{Internality of $J_\mM$  to $F$ in $\mathcal G$ means that there exists a $\CG$-definable injection from $J_\mM$ to $F^n$ for some $n$} of $J_\mM$ to $F$) will then reconstruct $J_\mM$ as a group over $F$. To that end, we first prove two lemmas aimed at guaranteeing internality of a relic to a field it interprets.

Given an algebraic group $(G,\cdot)$ over an algebraically closed field $K$, and an irreducible (locally closed) curve $X\subset G$, the \textit{stabilizer} of $X$ is the subgroup consisting of those $g\in G$ such that $(g\cdot X)\cap X$ is infinite (equivalently, $g\cdot \overline{X}=\overline{X}$). We denote the image of a set $Y\subset G$ in $G/\operatorname{Stab}(X)$ by $Y/\operatorname{Stab}(X)$. The first lemma says that in our setting, any failure of internality can be explained by a stabilizer:

\begin{sLemma}
\label{L: generically injective map}
Let $K$ be an algebraically closed field, $(G,\cdot)$ an algebraic group over $K$, and $\mathcal G=(G,\cdot,...)$ a $K$-relic. Assume that:
\begin{enumerate}
\item[(a)] $\mathcal G$ interprets the algebraically closed field $F$.
\item[(b)] $G$ is generated in finitely many steps by the $\mathcal G$-definable, irreducible, locally closed curve $X\subset G$.
\end{enumerate}
Then:
\begin{enumerate}
\item There is a $\mathcal G$-definable injection $f:X/\operatorname{Stab}(X)\rightarrow F^k$ for some $k$.
\item The group $G/\operatorname{Stab}(X)$ is internal to $F$ in the sense of $\mathcal G$.
\item In particular, if $\operatorname{Stab}(X)$ is trivial, then $\mathcal G$ is full.
\end{enumerate}
\end{sLemma}

\begin{proof} First, note that (1) implies (2) and (3). Indeed, a restatement of (1) is that $X/\operatorname{Stab}(X)$ is internal to $F$. But since $X$ generates $G$ in finitely many steps (and the same holds after quotienting by $\operatorname{Stab}(X)$), it follows that $G/\operatorname{Stab}(X)$ is also internal to $F$, and we get (2). Finally, if $\operatorname{Stab}(X)$ is trivial, we conclude from (2) that $G$ is internal to $F$, whereby fullness follows automatically by \cite{CasHas}. 

We now turn to (1). Since $X$ is an irreducible curve, it is strongly minimal in the sense of $\mathcal G$. Thus $G$ is internal to a strongly minimal set\footnote{In our setting strongly minimal sets in $\CG$ are $\CG$-definable irreducible locally closed curves in $G^n$ (any $n$).} (as it is generated by $X$ in finitely many steps), so that $\mathcal G$ is \textit{almost strongly minimal}. In particular, $G$ must now be \textit{almost internal} to $F$. We thus obtain a finite-to-one $\mathcal G$-definable map $f:G\rightarrow F^k$ for some $k$. 

In what follows, we will use the dimension theory of $K$ as an algebraically closed field (equivalently, Morley rank or Lascar rank specialized to $K$, as opposed to $\mathcal G$). Absorbing parameters, we assume $\mathcal G$ (and its basic relations), in addition to $F$ and $f$, are all $\emptyset$-definable in $K$.

Now let $g\in G$ be generic\footnote{Genericity is meant here in its model theoretic sense. In the present setting, this is the same as genericity in the sense of Weil, i.e., that $G$ is the locus of $g$.} (meaning, as above, in the sense of $K$). The main point is:

\begin{sClaim}
\label{Cl: gen-inj}
Let $a, b\in X\cdot g$ be two generics over $g$. If $f(a)=f(b)$, then $ab^{-1}\in\operatorname{Stab}(X)$. 
\end{sClaim}

\begin{claimproof}

Clearly $\dim(g,a,b)\geq\dim(g,a)=\dim(G)+1$. On the other hand, since $f(a)=f(b)$ and $f$ is finite-to-one, we have $b\in\acl(a)$, and thus $\dim(a,b)\leq \dim(G)$. So $g\notin\acl(a,b)$. Since $a,b\in X\cdot g$, we have $g^{-1}\in(a^{-1}\cdot X)\cap(b^{-1}\cdot X)$. So $a^{-1}\cdot X$ and $b^{-1}\cdot X$ have infinite intersection. Equivalently, $X\cap((ab^{-1}\cdot X)$ is infinite, so $ab^{-1}\in\operatorname{Stab}(X)$ as desired.
\end{claimproof}

We may now view $f$ as a set-valued function $\hat f$, sending each $a\cdot\opnm{Stab}(X)$ to its image (as a set) $f(a\cdot\operatorname{Stab}(X))\subset F^k$. In this language, a restatement of the above claim is that $\hat f$ is generically injective on $(X\cdot g)/\operatorname{Stab}(X)$ (meaning it is injective after removing finitely many points from the domain).

Note that $\operatorname{Stab}(X)$, the quotient $G/\operatorname{Stab}(X)$, and $\hat f$ are still $\mathcal G$-definable. In particular, applying elimination of imaginaries\footnote{In algebraically closed fields, elimination of imaginaries is equivalent to Weil's theorem on the existence of a smallest field of definition.} (in $F$) to the image of $\hat f$, we obtain a $\mathcal G$-definable generically injective function $h:(X\cdot g)/\operatorname{Stab}(X)\rightarrow F^m$ for some $m$. Finally, translation by $g$ gives a $\mathcal G$-definable bijection $X/\operatorname{Stab}(X)\rightarrow(X\cdot g)/\operatorname{Stab}(X)$. So, composing with $h$, we have a generically injective $\mathcal G$-definable function $X/\operatorname{Stab}(X)\rightarrow F^m$. By editing finitely many points, we can then replace $h$ by a genuinely injective function. Thus, we have shown (1), which proves the lemma.
\end{proof}

In light of Lemma \ref{L: generically injective map}, one is naturally driven to compute the subgroup $\operatorname{Stab}(U)\leq J_m$. Indeed, Zilber's argument crucially uses that a curve (of genus greater than 1) has trivial stabilizer in its Jacobian. In fact, the same holds for generalized Jacobians, and this fact really contains all geometric properties of $U$ and $J_m$ that we will use. 

\begin{sLemma}
\label{L: generating}
Let $\pi$ be the arithmetic genus\footnote{See, e.g., \cite[p.77]{SerreGJ}} of $U$. Then \begin{enumerate}
\item Every generic element of $J_\bm$ can be written uniquely as the sum of $\pi$ points of $U$ and, moreover, those points are generic in $U$ and independent. 
\item $J_\bm$ is generated by $U$ in $2\pi$ steps. 
\item For all $0\neq a\in J_\bm$ the intersection $(U+a)\cap U$ is finite. Thus, $\operatorname{Stab}(U)$ is trivial.
\end{enumerate}
\end{sLemma}

\begin{proof}
(2) follows from (1) since any element of $J_m$ is the sum of two generic elements.  (3) also follows from (1) by noticing that if for some $a$ the intersection were infinite (hence, co-finite), then for all generic $x_1,\dots, x_\pi\in U$ independent  over $a$ also $x_1-a$ and $x_2+a$ are in $U$. So $\sum_{i=1}^\pi x_i=(x_1-a)+(x_2+a )+x_3+\cdots + x_\pi$. Note that by $a$-independence and genericity of the $x_i$ over $a$ we get that $x_i\pm a\neq x_j$ for all $i,j$. Thus, the above contradicts the uniqueness part of (1). 

For (1) use Proposition 2 of \cite[\S.V]{SerreGJ} combined with Lemma 16 of the same section. Indeed, if $\theta$ is the Abel-Jacobi embedding of $U$ in $J_m$ then $\theta$ extends canonically to $U^{(\pi)}$ (identifying points in $U^{(\pi)}$ -- the orbits of $U^n$ under the natural action of $S_n$ -- with effective divisors of degree $n$ on $U$, the canonical extension is merely the sum of divisors in $J_\bm$). Then the above results show that this canonical extension coincides, generically, with a bi-rational map. So, in particular, it is generically bijective. 
\end{proof}

We have everything needed to prove Theorem \ref{T: Zil}:

\begin{proof}[Proof of Theorem \ref{T: Zil}]

First, note that $U$ is not (even up to finitely many points) a coset of a subgroup of $J_m$. Indeed, otherwise each of the sum sets $U$, $U+U$, $U+U+U$ ... would be one-dimensional, contradicting that $\dim(J_m)=\pi\geq 2$ and $U$ generates $J_m$ in finitely many steps (Lemma \ref{L: generating}). Model-theoretically, this means that $\mathcal J_m$ is not \textit{1-based} (by \cite{HrPi87}). So by Zilber's restricted trichotomy (\cite{CaHaYe}), $\mathcal J_m$ interprets an algebraically closed field, say $F$. 

We have now shown that $\mathcal J_m$ and $U$ satisfy the hypotheses of Lemma \ref{L: generically injective map}. By Lemma \ref{L: generating}, the group $\operatorname{Stab}(U)\leq J_m$ is trivial, so that we can in fact apply clause (3) of Lemma \ref{L: generically injective map}. We conclude that $\mathcal J_m$ is full, and then by Lemma \ref{L: full implies determines}, $U$ geometrically determines $J_m$. 
\end{proof}




We now make a few additional remarks on variants and generalizations of the above proof that may be of interest to model theorists.

\begin{sRemark} First, we note that the proof of Theorem \ref{T: Zil} can be understood more generally in the language of \textit{very ample types} from \cite{CasHas}. Recall that a non-algebraic stationary type $p$ in a stable theory is \textit{very ample} if for any two distinct realizations $a,b\models p$, $a$ and $\operatorname{Cb}(p)$ are dependent over $b$ (here $\operatorname{Cb}(p)$ is the canonical base of $p$). We leave it as an exercise to check the following, many of which also implicitly appear in \cite{CasHas} (for now we work in a superstable theory of finite rank):

\begin{enumerate}
    \item If $G$ is a definable group, and $p$ is a stationary type in $G$ of trivial stabilizer, then a generic translate of $p$ is very ample.
    \item If $p$ is a very ample stationary type over $A$, then no two realizations of $p$ are interalgebraic over $A$.
    \item Thus (by (2)), if $p$ is a very ample minimal type, and $p$ is almost internal (equivalently non-orthogonal) to a definable set $X$, then $p$ is actually internal to $X$.
    \item Combining (1) and (3) and translating by group elements appropriately we obtain: let $G$ be a definable group, $p$ a stationary type concentrated on $G$. Assume that  $p$ has a trivial stabilizer and is non-orthogonal to a definable set $X$.  Then $p$ is internal to $X$.
    \item In particular, in our setting, $U$ is internal to any infinite field non-orthogonal to it (including the field $F$ appearing in the proof), which implies that $U$ is full.
\end{enumerate}
\end{sRemark}

\begin{sRemark}
Our next observation is that we may view the data of the problem as determining a structure on $U$, not $J_\bm$. Indeed, since $J_\bm$ is generated by $U$, the structure $\mathcal J_\bm$ is interpretable in its induced structure on $U$ (see Lemma \ref{L: generating}). This means that in order to prove Theorem \ref{T: Zil}, we only need the trichotomy for \textit{one-dimensional} ACF relics, not arbitrary ones. In particular, to obtain the relevant definable field in the proof, we could use the (already published) main result of \cite{HaSu} in place of \cite{CaHaYe}.

\end{sRemark}

\begin{sRemark}
Finally, we note that the internality of the curve $U$ to the definable field generalizes to a different statement for curves with non-trivial stabilizer. Namely, the internality of $U$ can be deduced from Lemma \ref{L: generating} combined with the following, slightly more accurate statement: 
\end{sRemark}

\begin{sLemma}
Let $K$ be an algebraically closed field and $\CG$ an almost strongly minimal non-1-based $K$-relic expanding a connected commutative group $(G, +)$. Then there exist a $\CG$-definable field $F$, a finite subgroup $H$, and a $\CG$-definable $f:G\to F^k$  (some $k$) whose generic fibres are cosets of $H$. Moreover, if $C\sub G$ is a $\CG$-definable curve generating $G$, then the generic fibre of $f|C$ is a coset of $H\cap \{g\in G: |(g+C)\cap C|=\infty\}$. 
\end{sLemma}

Modulo the definability of the field $F$, the proof is, essentially, similar to \cite[Lemma 4.3]{ZilJac}, which in turn is a variant of an argument of Hrushovski's (appearing, e.g., in \cite[Porposition 2.25]{HrRid}. We leave the details as an exercise. 
\section{Proof of the conjecture}\label{S: proof}

\subsection*{Summary}

In this section we take the data provided by Theorem \ref{T: Zil} and convert it into our main theorem (precisely stated as Theorem \ref{BC33} and Proposition \ref{P: ACF version of main thm} below). The statement and proof of Theorem \ref{BC33} naturally involve the language of schemes; for the benefit of model theory oriented readers, we begin with a summary in concrete terms:

We work with the following setup:

\begin{itemize}
    \item $C,C'$ are geometrically irreducible smooth projective pointed curves over $\mathbb F_q$.
    \item $\mathfrak m,\mathfrak m'$ are effective divisors on $C,C'$ coprime to the base points.
    \item $U,U'$ are the complements of the supports of $\mathfrak m,\mathfrak m'$ in $C,C'$.
    \item $J_{\mathfrak m},J_{\mathfrak m'}$ are the associated generalized Jacobians, which we assume have dimension at least 2.
    \item To ease notation, we identify $U,U'$ with their images in $J_{\mathfrak m},J_{\mathfrak m'}$.
    \item $\psi:J_{\mathfrak m}(\overline{\mathbb F_q})\rightarrow J_{\mathfrak m'}(\overline{\mathbb F_q})$ is an isomorphism of abstract groups sending $U(\overline{\mathbb F_q})$ to $U'(\overline{\mathbb F_q})$.
\end{itemize}

In this language, Theorem \ref{T: Zil} says that we can factor $\psi$ as $f\circ\gamma$ where $\gamma$ is (induced by) an automorphism of $\overline{\mathbb F_q}$ and $f$ is a bijective isogeny defined over $\overline{\mathbb F_q}$. Our theorem (Theorem \ref{BC33}) will strengthen this set up in the following ways. For convenience, below we assume we have already factored out the automorphism $\gamma$, so that $J_{\mathfrak m}\rightarrow J_{\mathfrak m'}$ is already an isogeny.

\begin{enumerate}
    \item First, we show that (after potentially changing $\gamma$) we may take $f$ to be an \textit{isomorphism} of algebraic groups over $\overline{\mathbb F_q}$, not just a bijective isogeny (i.e. that $f^{-1}$ is also a regular map) -- in particular, this isomorphism of algebraic groups induces an isomorphism of subvarieties $U\rightarrow U'$, which in turn induces an isomorphism $C\rightarrow C'$. Note that this is similar to the reformulation of Zilber's original theorem in \cite{BoVo1}. Since Booher and Voloch do not provide a proof of their reformulation, we give the details.
    \item Next, we show that our isomorphism $C\rightarrow C'$ sends $\mathfrak m$ to $\mathfrak m'$ (thus, our isomorphism $J_{\mathfrak m}\rightarrow J_{\mathfrak m'}$ is precisely the one induced by the isomorphism of curves $C\rightarrow C'$).
    \item Finally, suppose we moreover know that $\psi$ sends $U(\mathbb F_{q^m})$ to $U'(\mathbb F_{q^m})$ for all $m$ (i.e. it induces a bijection on points over each intermediate field). In this case, we show that we can take $f$ to be defined over $\mathbb F_q$.
\end{enumerate}

Only (3) above really uses the finite field $\mathbb F_q$; (1) and (2) are proven working over $\overline{\mathbb F_q}$, and can be arranged over arbitrary algebraically closed fields (see Proposition \ref{P: ACF version of main thm}).

With goals (1)-(3) in mind, let us now briefly sketch the proof given in the rest of this section.

\begin{enumerate}
    \item First, our given bijective isogeny $J_{\mathfrak m}\rightarrow J_\mathfrak{m'}$ induces a bijective morphism $U\rightarrow U'$, which extends to a bijective morphism $C\rightarrow C'$. If $C\rightarrow C'$ is separable, it is an isomorphism (this is Lemma \ref{AB12}). Thus, our first step is to show that $C\rightarrow C'$ becomes separable after factoring out a suitable power of the Frobenius automorphism (this is Lemma \ref{AG14}; note that this is a special property of \textit{curves}: a map of higher-dimensional varieties does not always factor into a Frobenius power and a separable map). Since the Frobenius is a field automorphism, we may then absorb it into the given automorphism $\gamma$, and thereby assume $C\rightarrow C'$ is already an isomorphism. In particular, $U\rightarrow U'$ is now also an isomorphism, and it follows easily that $J_{\mathfrak m}\rightarrow J_{\mathfrak m'}$ is also an isomorphism (using that $U,U'$ generate $J_{\mathfrak m},J_{\mathfrak m'})$. This completes the first 
    point above. 
    \item Next we show the isomorphism $C\rightarrow C'$ sends $\mathfrak m$ to $\mathfrak m'$ (still working over $\overline{\mathbb F_q}$). Since the supports of $\mathfrak m,\mathfrak m'$ are the complements of $U$ and $U'$, we already know that $C\rightarrow C'$ sends $\operatorname{Supp}(\mathfrak m)$ bijectively to $\operatorname{Supp}(\mathfrak m')$. The point is that this bijection also preserves \textit{coefficients} -- i.e. if a point appears in $\mathfrak m$ with multiplicity $d$, then its image appears in $\mathfrak m'$ with multiplicity $d$. An equivalent statement (after cancelling the isomorphism between $C$ and $C'$) is that if $\mathfrak m,\mathfrak m'$ are effective divisors with the \textit{same} support on the \textit{same} pointed curve, and the identity map induces an isomorphism of their corresponding generalized Jacobians, then $\mathfrak m=\mathfrak m'$. This is amounts to a computation using the definition of the generalized Jacobian, and appears as Proposition \ref{AF12} and Corollary \ref{AF13}.
    \item Finally, we now suppose our isomorphism $\phi:C\rightarrow C'$ preserves $\mathbb F_{q^m}$-points for all $m$. We use an argument sketched to us by Booher to show that $\phi$ is defined over $\mathbb F_q$. The idea is to show (by counting fixed points of automorphisms) that $\phi$ commutes with any sufficiently large iterate of the $q$th power map, and that this property implies it is defined over $\mathbb F_q$. This argument appears as Lemmas \ref{MM44} and \ref{MM45}.
\end{enumerate}


We need some preliminaries. 

\subsection*{Notation and terminology}
For the benefit of model-theory-oriented readers, we review some standard notation and terminology from algebraic geometry. If $Z$ is a scheme, then a \emph{$Z$-scheme}, or a scheme \emph{over} $Z$ refers to a scheme $X$ equipped with a morphism of schemes $X \to Z$ which we refer to as the \emph{structure morphism of $X$}. A \emph{morphism of $Z$-schemes} is a morphism of schemes that commutes with the structure morphisms. When working with $Z$-schemes, it is often tacitly assumed that all morphisms are morphisms of $Z$-schemes. If instead we are given $Z$-schemes $X$, $Y$, and an endomorphism $\si: Z \to Z$, then a morphism of schemes $Y \to X$ is said to be \emph{$\si$-linear} if the square
\[
\xymatrix{
Y \ar[r] \ar[d] & X \ar[d]
\\
Z \ar[r]_\si & Z
}
\]
commutes.

Continuing with a $Z$-scheme $X$, if $T \to Z$ is another $Z$-scheme, then $X_T$ denotes the pullback, or \emph{base change} 
\[
X_T = T \times_Z X.
\]
By contrast, $X(T)$ denotes the set of morphisms of $Z$-schemes $T \to X$. We make frequent use of the canonical bijection $X(T) = X_T(T)$ where $X_T(T)$ refers to the set of morphisms of $T$-schemes.

If $X$ is an integral scheme, then the \emph{generic point} may be identified with a morphism of schemes $\Spec K \to X$ where $K = K(X)$ is the function field of $X$.

\spar{AD14}
We will be somewhat pedantic in distinguishing maps of rings from maps of associated schemes. For $q$ a power of $p$, we denote 
\[
F_q = \Spec \FF_q
\]
and 
$
\overline F_q = \Spec \overline \FF_q
$
the spectrum of an algebraic closure. We let $\Fr$ denote the absolute Frobenius morphism on any scheme of characteristic $p$. When working with a scheme $X$ over a base $Z$ of characteristic $p$ we let $(\Fr^m)^* X$ denote  the pullback of $X$ along the $m$th power of the absolute Frobenius map of $Z$. When the latter is an isomorphism, we allow $m < 0$. We denote the associated relative Frobenius morphism by the two alternative notations
\[
\Fr^m_\m{rel} = \Fr^m_{/Z}: X \to (\Fr^m)^*X. 
\]

\spar{AD15} We fix an algebraic closure $\overline \FF$ of $\FF_q$ and we denote $\Fbar := \Spec \overline \FF$. For $r$ a power of $q$, we let $\FF_r \subset \overline \FF$ denote the subfield of $r$ elements and as usual $F_r := \Spec \FF_r$. 

If $\be: \Fqbar \to \Fqbar$ is an automorphism,  $\FF_q \subset k \subset \overline \FF$ is an intermediate field, and $S = \Spec k$, we denote the induced automorphism of $S$ by $\be_S$. When $k = \FF_r$, we also write $\be_r := \be_{F_r}$ (equal to $\Fr^m$ for some $m \in \ZZ$). 

Suppose $X \to \Fq$ is a scheme over $\Fq$, $\be: \Fqbar \to \Fqbar$ is an automorphism, and let $S$ be as above. Then $\be$ induces a map
\[
\be_S^X: X(S) \to 
(\be_q^*X \big) (S).
\]
Namely, if $x: S \to X$ is an $S$-point over $\Fq$, then the composition
\[
S \xto{\be_S} S \xto{x} X
\]
is $\be_q$-linear, hence induces an $\Fq$-linear map
\[
\be^X_S(x): S \to \be_q^*X.
\]
In the special case $k = \FF_r$, we write also $\be^X_r:= \be^X_S$. Further, when $r = 1$ (i.e. $\FF_r = \FF_q$) we allow ourselves to drop the subscript. 

In the special case that $\be = \Fr: \Fqbar \to \Fqbar$ is the absolute Frobenius morphism, it follows from the fact that Frobenius commutes with all morphisms that the action $\be^X$ of $\be$ on $X(\Fqbar)$ as above coincides with the action of the relative Frobenius morphism $\Fr_{/F_q}$. We spell this out too. It is enough to check that the lower left triangle in the following diagram 
\[
\xymatrix{
\Fqbar \ar[d]_x \ar[r]^\Fr \ar[dr]^{\be^X(x)}  & \Fqbar \ar[r]^x & X
\\
X \ar[r]^-{\Fr_{/F_q}} \ar@/_7ex/[rru]_\Fr & \Fr^*X \ar[ur]_W^\sim
}
\]
commutes. Since $W$ (the pullback of $\Fr: \Fq \to \Fq$ along $X \to \Fq$) is an isomorphism (never mind that it is not $\Fq$-linear), it is enough to check that $W \circ \Fr_{/F_q} \circ x = W \circ \be^X(x)$, which follows from the commutativity of the square with edges $x$, $x$, $\Fr$, and $\Fr$. 

We now depart from our finite field $\FF_q$ for a number of general lemmas that will be needed in the proof of Theorem \ref{BC33}. The reader may wish to skip to paragraph \ref{AE12} and refer back as needed.


\begin{sLemma}
\label{AB12}
Let $g: C \to C'$ be a separable morphism of smooth proper curves over a perfect field $k$. Suppose there are open subschemes $U \subset C$, $U' \subset C'$ such that $g$ induces a bijection $U(\kbar) \to U'(\kbar)$ for some fixed algebraic closure. Then $g$ is an isomorphism. 
\end{sLemma}

\begin{proof}

Since $g$ is separable, generically, the fibers of $g$ are reduced. This means on the one hand that there exists a closed point $y \in U'$ such that $g\inv(y)$ is reduced (actually most points), and on the other hand, that if $\eta' = \Spec K'$ is the generic point of $C'$, then $g\inv(\eta')$ is reduced to the generic point $\eta = \Spec K$ of $C$. (For the second statement, a pullback of a localization morphism is a localization morphism; a localization of an integral ring is integral).

Thus, $g\inv(y) \to y$ is the spectrum of a reduced Artin algebra $A$. Moreover, all but finitely many $y \in U'$ have $g \inv(y) \subset U$. Thus, we may assume $g\inv(y)$ has only one $\kbar$-point, hence $A$ has dimension $1$ over $k(y)$.

Any nonconstant morphism of smooth curves is faithfully flat. In particular, the degree of fibers over field-valued points is constant. Combining with the above, we find that $K'/K$ is a field extension of degree 1. Hence $g$ is an isomorphism. 
\end{proof}

As a matter of notation, if $D = \sum_x n_x [x]$ is a divisor on a curve, we denote the coefficient $n_x$ by
\[
\ord_x D.
\]

\begin{sLemma}
\label{AG10}
Let $X$ be a smooth 
curve over an algebraically closed field $k$, let $S$ be a finite set of closed points of $X$, and let
$
\mM
$
be a divisor supported in $S$. Then there exists a rational function $f \in K(X)$ such that for each $x \in S$,
\[
\ord_x f = \ord_x \mM. 
\]
\end{sLemma}

\begin{proof}
It is enough to fix $x$ in $S$ and to construct a rational function $g$ such that $\ord_x g = 1$ and $\ord_y g = 0$ at all other points $y$ of $S$ --- the requited function $f$ is then a product of such functions. The abundance of rational functions on $X$ can be expressed, in a way useful for the construction of $g$, by embedding $X$ in a projective space $P$. There exist hyperplanes $H_1$ and $H_2$ such that $H_1$ contains $x$ but $H_1$ is not tangent to $X$ at $x$ and $H_1$ does not contain any other points of $S$, while $H_2$ does not contain any points of $S$. We may assume $X$ does not lie in either $H_1$ or $H_2$, so if $G$ is a rational function on $P$ with divisor $[H_1] - [H_2]$, then $g := G|_X$ does the job.
\end{proof}

\begin{sProposition}
\label{AF12}
Let $k$ be a field, $C$ a smooth proper curve over $k$, $x_0$ a $k$-rational point, $S$ a finite set of closed points disjoint from $x_0$ with complement $U$, $\mM$, $\mM'$ effective divisors with common support \emph{equal} to $S$, and let $j_\mM$, $j_{\mM'}$ denote the embeddings $U \rightrightarrows J_\mM, J_{\mM'}$ of $U$ in the associated generalized Jacobian varieties. Suppose there exists a morphism $\phi$ of algebraic groups such that the triangle
\[
\xymatrix{
& J_{\mM'} \ar[dd]^-{\phi}
\\
U \ar[ur]^-{j_{\mM'}} \ar[dr]_-{j_{\mM}}
\\
& J_{\mM}
}
\]
commutes. Then $\mM' \ge \mM$.
\end{sProposition}

\begin{proof}
We may assume $k$ algebraically closed and $S \neq \emptyset$. The group $J_\mM(k)$ may be identified with the group $\Div^{0}(U)/P_\mM$ of degree $0$ divisors coprime to $S$ modulo divisors of meromorphic functions congruent to $1 \opnm{mod} \mM$. In terms of this identification, the embedding $j_\mM$ is given by
\[
j_\mM(x) = [x] - [x_0],
\]
and similarly for $j_\mMp$. If we restrict attention to $k$-points and extend $j_\mM$, $j_\mMp$ to divisors by linearity, the triangle
\[
\xymatrix{
&& J_\mMp(k) \ar[dd]^-\phi
\\
\Div^0(U) \ar@{}[r]|\subset
&
\Div(U) \ar[ur]^-{j_\mMp} \ar[dr]_-{j_\mM}
\\
&&
J_\mM(k)
}
\]
still commutes. Moreover, if the divisor $D=\sum_n a_n x_n$ has degree $0$ ($\sum_n a_n = 0$), then 
\[
j_\mM(D) = \sum_n a_n([x_n] - [x_0]) = \sum_n a_n [x_n]
\]
and similarly for $j_\mMp$.

By lemma \ref{AG10} there exists a meromorphic function $f \in K(C)^*$ such that $\opnm{div}(f) = \mMp + \mathfrak{r}$ with $\mathfrak{r}$ coprime to $S$; in particular, $\opnm{div}(f+1) \equiv 1 \mod \mMp$ so that $j_\mMp\big(\opnm{div}(f+1)\big) = 0$. Since
\[
j_\mM \big( \opnm{div}(f+1) \big)
=
\phi \big(j_\mMp \big( \opnm{div}(f+1) \big)\big)
=
0,
\]
$\opnm{div}(f+1)$ is the divisor of a meromorphic function $g+1$ such that $\opnm{div}(g+1) \equiv 1 \mod \mMp$. 

On the one hand, this means that $\opnm{div}(g) = \mM'' + \mathfrak{q}$ with $\mM'' \ge \mM$ and $\mathfrak{q}$ coprime to $S$. On the other hand, the equality of divisors
\[
\opnm{div}(f+1) = \opnm{div}(g+1)
\]
implies that $f$ and $g$ differ by a constant, hence have the same divisor.\footnote{In fact, the constant is $1$ since $f(x) = g(x) = 0$ for any $x \in S$.} Together, we find that 
\[
\mM'' + \mathfrak{q} = \opnm{div}(g) = \opnm{div}(f) = \mMp + \mathfrak{r}
\quad \text{with }
\mathfrak{q}, \mathfrak{r} \text{ coprime to }S,
\]
and hence that 
\[
\mM' = \mM'' \ge \mM.
\qedhere
\]

\end{proof}

\begin{sCorollary}
\label{AF13}
In the situation of Proposition \ref{AF12}, suppose $\dim J_\mM = \dim J_\mMp$. Then $\mM = \mMp$ and $\phi$ is the identity map. 
\end{sCorollary}

\begin{proof}
We may again assume $k$ is algebraically closed. If $\mMp > \mM$ then by \cite[\S V.1.6 Proposition 2]{SerreGJ} $\dim J_\mMp > \dim J_\mM$, so the first statement follows from Proposition \ref{AF12}. The second statement then follows from the surjectivity of $j_\mM |_{\Div^0(U)}$.
\end{proof}

\begin{sLemma}
\label{AG13}
Let $k$ be a perfect field of positive characteristic $p$ and $L \overset{\iota}\supset K$ be a purely inseparable finite extension of fields of transcendence degree $1$ over $k$. Assume $t \in K$ is a transcendental element over $k$ and $[K:k(t)]$ is finite. Then $\iota$ is given by a relative Frobenius morphism over $k$. That is, there exists an $m \in \Nn$ such that if $L^{(m)} = k \otimes_k L$ denotes the twist of $L$ by the $m$th power of the Frobenius morphism of $k$, then the induced map $L \from L^{(m)}$ has image equal to $K$. 
\end{sLemma}

\begin{proof}
By assumption there is some $n$ such that $[L:K]=p^n$. We first show, that since $\mathrm{tr.deg}(K/k)=1$ and $k$ is perfect, we get that $[K^{(1/p^n)}: K]=p^n$. To see that, it will suffice to show that $[K^{(1/p^n)}: K]=[k(t)^{(1/p^n)}:k(t)]$ since the latter is obviously equal to $p^n$. But $[K: k(t)]=[K^{(1/p^n)}: k(t)^{(1/p^n)}]$, and 
\[
[K^{(1/p^n)}: k(t)]=[K^{(1/p^n)}: k(t)^{(1/p^n)}][k(t)^{(1/p^n)}:k(t)]
\]
while also 
\[
[K^{(1/p^n)}: k(t)]=[K^{(1/p^n)}:K][K:k(t)].
\] Putting everything together, we get, indeed, that $[K^{(1/p^n)}: K]=p^n$.

It is standard to check (see, e.g, the proof of \cite[Lemma 4.15]{HaSu}) that there is some $m$ such that $\Fr^m(L)\subseteq K$. Let us verify that if $m$ is minimal such, then $\Fr^m(L)=K$ (and $m=n$). The choice of $m$ assures, on the one hand, that $K^{(1/p^m)}\supseteq L$, so by the first paragraph, $m\ge n$. But the minimality of $m$ assures that $[L:K]\ge p^m$, so $m\le n$. Thus, $m=n$, implying that $[K^{(1/p^n)}:L]=1$, with the desired conclusion. 
\end{proof}

\begin{sLemma}
\label{AG14}
Let $f: X \to Y$ be a map of smooth proper curves over a perfect field $k$. Then $f$ factors as 
\[
X \xto{\Fr^m_\m{rel}} (\Fr^m)^* X \xto{f'} Y
\]
with $f'$ separable, for some $m \in \Nn$. 
\end{sLemma}

\begin{proof}
The functor $K(\cdot)$ from the category of smooth proper curves over $k$ to fields of transcendence degree $1$ over $k$ is an equivalence of categories. Moreover, $K(\cdot)$ is compatible with formation of twists by relative Frobenius and relative Frobenius morphisms. Hence, this follows from \ref{AG13} along with the fact that any field extension may be factored as a separable extension followed by a purely inseparable extension.
\end{proof}

\begin{sLemma}
\label{MM44}
Let $C$ be a smooth proper curve over a field $k$, and let $\mM$ be an effective divisor. Assume $\dim J_\mM \ge 2$. For $\rho$ an automorphism of $C$ such that $\rho^*\mM = \mM$ we let $C(k)^\rho$ denote the set of $k$-valued fixed points. Then
\[
\{\sharp C(k)^\rho \,|\, \rho \neq id\}
\]
is bounded.

\end{sLemma}

\begin{proof}
We may assume $k$ is algebraically closed. Any non-identity automorphism has only finitely many fixed points. If the genus $g$ of $C$ is  at least $2$, then $\Aut(C)$ is finite (see, e.g., \cite[Exercise V.1.11]{Hartshorne}), and the result follows, so we may assume $g \le 1$. There is an evident map
\[
\pi:\Aut(C, \mM) \to \Aut(S)
\]
to the symmetric group on the support $S$ of $\mM$. If $g=1$, then our assumption on the dimension of $J_\mM$ implies that $\mM \neq 0$, and so, fixing a point $c$ in the support of $\mM$, we find that $\ker \pi$ embeds in $\Aut(C,c)$ which is again known to be finite (e.g., \cite[Theorem III.10.1]{Silverman}), and the result follows. 

Finally, in the case $g=0$, $\Aut(C, \mM)$ embeds in $\Aut (\PP^1) \simeq PGL_2$, and its elements have at most two fixed points.
\end{proof}

We learned about the following lemma from Jeremy Booher. 

\begin{sLemma}
\label{MM45}
Let $(C, \mM)$, $(C', \mM')$ be smooth proper curves over $F_q$ equipped with non-zero effective divisors with complements $U$, $U'$, and let
\[
\phi:
(C, \mM)_{\bar F} 
\to 
(C', \mM')_{\bar F}
\]
be an isomorphism, i.e. an isomorphism $C_\Fbar \to C'_\Fbar$ such that $\phi^*\mM'_\Fbar = \mM_\Fbar$. Suppose that for every $m$, $\phi$ induces a bijection $U(F_{q^m}) \to U'(F_{q^m})$. Then $\phi$ is defined over $F_q$. 
\end{sLemma}

\begin{proof}
Let $\Fr_{q^m} = \Fr_q^m$ denote the $q^m$-power Frobenius map of any scheme of characteristic $p$. We have $\Fr_q^* \mM = q\mM$, and similarly for $\mM'$. Thus, $\rho_m := \Fr_{q^m}^{-1} \phi^{-1} \Fr_{q^m}  \phi$ is an automorphism of $(C, \mM)_\Fbar$ which fixes $U(F_{q^m})$. By lemma \ref{MM44}, for $m$ sufficiently large,
\[
\sharp U(F_{q^m}) > \max_{\rho \neq id} \sharp C(\Fbar)^\rho.
\]
For such $m$, $\rho_m = id$, so $\phi$ commutes with $\Fr_{q^m}$ and so $\phi$ is defined over $F_{q^m}$. Varying $m$, we find that, by our assumption, $\phi$ is defined over the intersection of the fields $F_{q^m}$ for all sufficiently large $m$. The latter is equal to $F_q$. 
\end{proof}

\spar{AE12} 
We fix two smooth projective irreducible pointed curves $C$, $C'$ over $F_q$ and $\mM$, $\mM'$ effective divisors whose supports are disjoint from the basepoints. We assume the generalized Jacobians $J_\mM$, $J_{\mM'}$ have dimension at least $2$. Using the base-points, we may regard the complements $U$, $U'$ of the supports of the divisors $\mM$, $\mM'$ as embedded in the respective generalized Jacobians $J_\mM$, $J_{\mM'}$. For $m \in \mathbb{Z}$, we write
\[
(\Fr^m)^* \mM = W_m^* \mM
\]
for the pullback of $\mM$ along the isomorphism
\[
W_m: (\Fr^m)^*C \to C
\]
obtained by pullback from 
$
\Fr^m: F_q \to F_q
$
along the structure morphism
$
C \to F_q.
$
We endow $(\Fr^m)^*C$ with the modulus $(\Fr^m)^* \mM$.

\spar{XY55}
Let $\al$ be an isomorphism
\[
(\Fr^m)^*C \to C'
\]
of $\Fq$-curves with modulus, and let $\be$ be an automorphism of $\Fbar$. We let $\tilde \al$ denote the induced isomorphism of generalized Jacobians
\[
\tilde \al:J_{(\Fr^m)^*\mM} = (\Fr^m)^*J_\mM \to J_{\mM'}.
\]
Taking $\Fqbar$-points and composing with $\be^{J_\mM}$ (\ref{AD15}), we obtain an isomorphism of abstract groups
\[
J_\mM(\Fqbar)
\xto{\be^{J_\mM}}
\big( ( \Fr^m)^* J_\mM \big)
(\Fqbar)
\xto{\tilde \al(\Fqbar)}
J_{\mM'}(\Fqbar).
\]
Moreover, the isomorphism $\tilde \al(\Fqbar) \circ \be^{J_\mM}$ commutes with the respective embeddings of $U(F_r)$, $U'(F_r)$ for every $r$. 
Theorem \ref{BC33} states that any isomorphism of groups of $\Fqbar$-points of generalized Jacobians commuting with the embeddings is of this form.


\begin{sTheorem}
\label{BC33}
In the situation and the notation of paragraph \ref{AE12}, suppose 
\[
\psi: J_\mM(\Fqbar) \to J_{\mM'} (\Fqbar)
\]
is an isomorphism of abstract groups such that for every power $r$ of $q$, 
\[
\psi \big(U(F_r)\big) = U'(F_r).
\]
Then there exists an automorphism $\be: \Fbar \to \Fbar$ and an isomorphism 
\[
\al: \be_q^* C \to C'
\]
of curves over $F_q$ such that 
\[
\tag{*}
\al^* \mM' = \be_q^* \mM
\]
and such that 
\[
\tag{**}
\psi = \tilde \al(\Fqbar) \circ \be^{J_\mM}.
\]

\end{sTheorem}

The proof spans paragraphs \ref{VD44}--\ref{VG56}. 

\spar{VD44}
By Theorem \ref{T: Zil}, there exists an automorphism $\ga : \Fqbar\xto{\sim} \Fqbar$ and a bijective isogeny of algebraic groups over $\Fbar$
\[
\tilde f: \ga^*(J_{\mM, \Fqbar}) = (\ga_q^*J_\mM)_{\Fqbar} \to J_{\mM',\Fqbar}
\]
(i.e. one that induces a bijection on $\Fqbar$-points) such that the composition
\[
\tag{$\psi$}
J_\mM(\Fqbar) 
\xto{\ga_{J_\mM}} 
(\ga_q^*J_\mM)(\Fqbar) 
\xto{\tilde f(\Fqbar)} 
J_{\mM'}(\Fqbar)
\]
is equal to $\psi$. Our first goal will be to construct the automorphism $\be$, as well as an isomorphism 
\[
\al_\Fbar: \be^*C_\Fbar \to C'_\Fbar
\]
of curves over $\Fbar$
which will later be shown to be defined over $\Fq$. 

\spar{VD45}
Since the composition $\tilde f \circ \iota$ below
\[
\xymatrix{
\ga^*J_{\mM, \Fqbar}
\ar[r]^-{\tilde f}
&
J_{\mM', \Fqbar}
\\
\ga^*(U_\Fqbar)
\ar@{^(->}[u]^-\iota
\ar@{.>}[r]_-f
&
U'_\Fqbar
\ar@{^(->}[u]
}
\]
carries the $\Fqbar$-points of $\ga^*(U_\Fqbar)$ to the $\Fqbar$-points of $U'_\Fqbar$, we find that $\tilde f$ restricts to a (unique) map $f$ of curves over $\Fbar$, as shown. This map extends uniquely to a map of complete smooth curves $ \ga^*(C_\Fqbar) \to C'_\Fqbar$, which we continue to denote by $f$. By Lemma \ref{AG14}, we may factor $f$ as 
\[
\tag{$f$}
\ga^*(C_\Fqbar) 
\xto{\Fr_{/\Fbar}^l} 
(\Fr^l)^* \ga^*(C_\Fqbar) 
\xto{\al_\Fbar} 
C'_\Fbar
\]
for some $l$, with $\alpha_\Fbar$ separable. By Lemma \ref{AB12} it follows that $\al_\Fbar$ is an isomorphism. We complete the promised construction by setting $\be = \ga \circ \Fr^l$.

\spar{VD46}
To verify \ref{BC33}(*) and (**), we will need the following observation. Fix an intermediate field $\FF_q \subset \FF \subset \overline \FF$ with spectrum $F = \Spec \FF$, and write $\ga_F$ for the automorphism of $F$ induced by $\ga$. The map $\ga_{J_\mM}$ fits in an evident commuting square (lower left)
\[
\xymatrix{
J_\mM ( \Fbar )
\ar@/^4ex/[rr]^-{\psi}_-\sim 
\ar[r]_-{\ga_{J_\mM}}
&
(\ga_q^*J_\mM)(\Fbar)
\ar[r]_-{\tilde f (\Fbar)}
&
J_{\mM'}(\Fbar)
\\
U(F)
\ar@{}[u]|\bigcup
\ar[r]_-{(\ga_F)_U}
&
(\ga_q^*U)(F)
\ar@{}[u]|\bigcup
\ar@{.>}[r]_{f(\Fbar)}^-\sim
&
U'(F).
\ar@{}[u]|\bigcup
}
\]
It follows that $f$ induces a bijection, as shown.

\spar{QA33}
Next, we show that the equality of divisors \ref{BC33}(*) holds at least over $\Fbar$. Since $\al_\Fbar$ is an isomorphism of complete curves, the composition \ref{VD45}(f) induces a bijection of $\Fbar$-points. We have seen that $f$ restricts to a bijection of open curves. It follows that $f$ also restricts to a bijection of the complements $\Supp \ga^* \mM_\Fbar \xto{\sim} \Supp \mM'_\Fbar$, and hence the divisors $\al_\Fbar^* \mM'_\Fbar$ and $\be^* \mM_\Fbar$ on $\be^* C_\Fbar$ have equal support. Applying Corollary \ref{AF13}, we find that 
\[
\al_\Fbar^* \mM'_\Fbar=\be^* \mM_\Fbar.
\]

\spar{VE56}
We now argue that $\al_\Fbar$ is defined over $\Fq$. Returning to the composition \ref{VD45}(f), and to the intermediate field $F = \Spec \FF$, we have, on the one hand, that the composition induces a bijection of $F$-points of open curves. On the other hand, the relative Frobenius morphism of $\ga^*U_\Fbar$ over $\Fbar$ is the base-change of the relative Frobenius morphism of $\ga_q^*U$ over $F_q$, and the latter induces a bijection on $F$-points. So the first morphism too induces a bijection on $F$-points of open curves, and so finally, $\al_\Fbar$ induces a bijection on $F$-points of open curves. By Lemma \ref{MM45}, it follows that $\al_\Fbar$ is defined by a map of curves $\al: \be_q^*C \to C'$ over $F_q$. 

\spar{VF55}
Having previously established the equality of divisors \ref{BC33}(*) over $\Fbar$, it follows that the same holds already over $F_q$. 

\spar{VG56}
Finally, to verify \ref{BC33}(**), we may go back to working over $\Fbar$. We've seen above that $\al$ is a map of curves with modulus, so we may consider the induced map
\[
\tilde \al: \be_q^*J_\mM \to J_{\mM'}
\]
of generalized Jacobians. We claim that the triangle 
\[
\xymatrix{
\ga^* J_{\mM, \Fbar}
\ar[d]_-{\Fr^l_{/\Fbar}}
\ar[dr]^-{\tilde f}
\\
\be^*J_{\mM, \Fbar}
\ar[r]_-{\tilde \al_\Fbar}
&
J_{\mM', \Fbar}
}
\]
commutes. Indeed, it's enough to check the commutativity on the $\Fbar$-points of $\ga^*U_\Fbar$, which holds by construction. This completes the proof of Theorem \ref{BC33}.

\spar{fff}
A less precise statement holds over an arbitrary field $k$. This statement factors through an algebraic closure of $k$, so we may as well assume $k$ itself to be algebraically closed. 

\begin{sProposition}\label{P: ACF version of main thm}
Let $(C,\mM, U)$, $(C', \mM', U')$ be pointed curves with moduli over an algebraically closed field $k$, and
\[
\psi: J_\mM(k) \xto{\sim} J_{\mM'}(k)
\]
an isomorphism of abstract groups inducing a bijection
\[
\psi: U(k) \xto{\sim} U'(k).
\] 
Then there exists an automorphism $\be : \Spec k \xto{\sim} \Spec k$ and an isomorphism of pointed curves
\[
\al: \be^* C \to C'
\]
such that $\psi$ is induced by $\be$ followed by $\al$, and
\[
\be^*\mM = \al^* \mM'. 
\]
\end{sProposition}

\begin{proof}
By theorem \ref{T: Zil} there exists an automorphism $\ga: \Spec k \to \Spec k$ and a bijective isogeny
\[
f: \ga^*J_\mM \to J_{\mM'}
\]
such that the composition 
\[
\xymatrix{
(\ga^*J_\mM)(k)
\ar[r]^f
&
J_{\mM'}(k)
\\
J_\mM(k)
\ar[u] \ar[ur]_\psi
}
\]
is equal to $\psi$, as shown. The map $f$ induces maps of open / closed curves 
\[
\xymatrix{
\ga^* J_\mM \ar[r]^f & J_{\mM'}
\\
\ga^*U \ar[r]_g \ar@{^(->}[u] \ar@{^(->}[d]
& 
U' \ar@{^(->}[u] \ar@{^(->}[d]
\\
\ga^*C \ar[r]_g  & C'.
}
\]
If the characteristic of $k$ is $0$, then $f$ and $g$ are isomorphisms, and it follows that
\[
g^* \mM' = \ga^*\mM
\]
as in the proposition.

Suppose now that the characteristic of $k$ is positive. By lemmas \ref{AG14} and \ref{AB12}, $g$ may be factored as
\[
\ga^* C \xto {\Fr_\m{rel}^l} 
(\Fr^l)^* \ga^* C \underset{\sim}{\xto{\al}}
C',
\]
where $\al$ is an isomorphism. If we endow the pointed curve $(\Fr^l)^* \ga^* C$ with the modulus $(\Fr^l)^* \ga^* \mM$, then $\al$ induces a map of generalized Jacobians, and we have
\[
\dim J_{(\Fr^l)^* \ga^* \mM} = \dim J_{\mM'}.
\]
It follows from Proposition \ref{AF13} that 
\[
(\Fr^l)^* \ga^* \mM = \al^*\mM,
\]
which completes the proof of the proposition.
\end{proof}

\section{Recovering open curves from L-fuctions}

As mentioned in the introduction, Booher and Voloch \cite{BoVo1} show that Theorem \ref{BC33} implies Corollary \ref{BA36} below. Their treatment is somewhat condensed, and we include a detailed account here for the reader's convenience. 

\spar{BA33}
Let $(C,x)$ be a smooth proper pointed curve over the field $\Fq$ of $q = p^n$ elements, let $K = K(C)$ be the function field, and $\mM$ an effective divisor coprime to $x$, let $U = C \setminus \Supp \mM$, and let 
\[
j_\mM: U \to J_\mM
\]
be the canonical embedding of $(U, x)$ in the generalized Jacobian associated to $\mM$. For any $\Fq$-scheme $X$, we let $\Fr_q = \Fr^n: X \to X$ be the $\Fq$-linear Frobenius morphism. Let $\tilde U \to U$ be the finite \'etale cover obtained from the pullback
\[
\xymatrix{
\tilde U \ar[r] \ar[d] & J_\mM \ar[d]^-{\Fr_q-\mathrm{id}}
\\
U \ar[r] & J_\mM
}
\]
and let $\tilde C \to C$ be the induced map of smooth proper curves over $\Fq$. Let $K^\mM_x$ be the function field of $\tilde U$. Then according to Proposition 2.2 of Booher-Voloch \cite{BooVol},
$K^\mM_x$ is an abelian Galois extension of $K$, and there is an isomorphism
\[
\tag{$\ast$}
\Theta: \Gal(K^\mM_x/K) \xto{\sim} J_\mM(\Fq)
\]
with the following property: if $y$ is a closed point of $\tilde U$ of degree $\deg(y)$ over its image in $C$, then
\[
\opnm{Frob}_y \overset{\Theta}\mapsto [y - \deg(y)x].
\]
(See Chapter VI, Section 5, Subsection 22 of Serre \cite{SerreGJ} for the \textit{Frobenius substitution} $\opnm{Frob}_y$.)

\spar{BA34}
Continuing with the situation and the notation of paragraph \ref{BA33}, 
if
\[
\chi:\Gal(K^\mM_x/K) \to \CC^*
\]
is a homomorphism, then the L-fuction of $\chi$ is defined by
\[
L(T, C, \chi) :=
\prod_{y \in |U|} \big( 1- \chi(\opnm{Frob}_y) T^{\deg(y)} 
\big)^{-1}
\]
where the product ranges over the set $|U|$ of closed points of $U$. Translating via the isomorphism \ref{BA33}(*) (and importing the proof of the classical Euler identity from e.g. Chapter VII Proposition 1.1 of Neukirch \cite{Neukirch}), we find that the L-function may be written as a sum over effective divisors on $U$ like so:
\[
\tag{*}
L(T, C, \chi) = \sum_{D \in \opnm{Div}_{\ge 0}(U)}
\chi \big(
[D - \deg(D)x]
\big)
T^{\deg(D)}.
\]

\spar{BA35}
We return to the situation of paragraph \ref{AE12} given by two smooth projective irreducible pointed curves $(C,c)$, $(C', c')$ over $\Fq$ and $\mM$, $\mM'$ effective divisors whose supports are disjoint from the base-points. We let $U = C\setminus \Supp(\mM)$ and similarly for $U'$. We assume the generalized Jacobians have dimension at least $2$. For any finite extension $F_{q^m} = \Spec \FF_{q^m}$ of $F_q$, the construction of paragraph \ref{BA33} supplies a finite abelian Galois extension $K(C_{F_{q^m}})^\mM_x$ of the function field $K(C_{F_{q^m}})$ of the base-change $C_{F_{q^m}}$ of $C$ to $F_{q^m}$, as well as an isomorphism
\[
\tag{$\ast_n$}
\Gal\big(
K(C_{F_{q^m}})^\mM_x / K(C_{F_{q^m}})
\big)
\xto{\sim}
J_\mM(F_{q^m}).
\]
Thus, a homomorphism 
\[
\chi: J_\mM(F_{q^m}) \to \CC^*
\]
may be thought of as a Galois character, and we may reasonably speak of the associated L-function $L(T, C_{F_{q^m}}, \chi)$, and similarly for $C'$.

As above, we denote the Abel-Jacobi maps by $j_\mM$, $j_{\mM'}$. As Booher and Voloch explain, the following statement follows from Theorem \ref{BC33}.

\begin{sCorollary}[Conjecture 2.4 of Booher and Voloch \cite{BoVo1}]
\label{BA36}
In the situation and the notation of paragraph \ref{BA35}, suppose the homomorphism of abstract groups
\[
\psi: J_\mM(\Fqbar) \to J_{\mM'}(\Fqbar)
\]
induces an isomorphism 
\[
J_\mM(F_{q^m}) \simeq J_{\mM'}(F_{q^m})
\]
for every $m \ge 1$. For any character $\chi$ of $J_\mM(F_{q^m})$, we denote by
\[
\chi' := \chi \circ \psi^{-1}|_{J_{\mM'}(F_{q^m})}
\]
the associated character of $J_{\mM'}(F_{q^m})$. If for all $m \ge 1$ and all characters $\chi$ of $J_\mM(F_{q^m})$, 
\[
L(T, C_{F_{q^m}}, \chi) = 
L(T, C'_{F_{q^m}}, \chi' ),
\]
then there exists an automorphism $\be: \overline F_q \to \overline F_q$ (inducing $\Fr^l$ on $F_q$ for some $l \in \mathbb{Z}$), and an isomorphism 
\[
\al: (\Fr^l)^* C \to C'
\]
of curves over $F_q$ such that 
\[
\al^* \mM' =(\Fr^l)^* \mM
\]
and such that (in the notation of paragraph \ref{XY55})
$
\psi = \tilde \al \circ \be^{J_\mM}.
$

\end{sCorollary}

The proof hinges on the following lemma. 

\begin{sLemma}
\label{BB33}
Let $P_\mM \subset K(C)^*$ denote the subgroup of divisors of functions congruent to $1$ modulo $\mM$, and write $\equiv_\mM$ for equivalence of divisors coprime to $\mM$ modulo $P_\mM$. Let $E \in \Div(U)$ be a divisor of degree $0$ coprime to $\mM$. Then $E+c$ is congruent modulo $P_\mM$ to an $F_q$-rational point of $U$ if and only if 
\[
\tag{*}
\sharp 
\Big\{ D \in \Div_{\ge 0}(U) 
\; \big| \; 
D \equiv_\mM E +c 
\Big\} = 1.
\]
\end{sLemma}

\begin{proof}
Evidently, equivalence modulo $P_\mM$ \textit{implies} equivalence modulo $K(C)^*$, i.e. linear equivalence. Since $E+c$ has degree $1$, any effective divisor $D$ linearly equivalent to $E+c$ must be an $F_q$-point $D = d \in U(F_q)$. Moreover, $d \equiv_\mM E+c$ if and only if $d\in j_\mM^{-1}(E)$. Thus, (*) is equivalent to 
\[
\tag{**}
\sharp\big(j_\mM^{-1}(E)\big) = 1. 
\]
Our assumption that $\dim J_\mM \ge 2$ guarantees that $j_\mM$ is injective. 
\end{proof}

\spar{BB34}
The proof of Corollary \ref{BA36} also makes use of the classical fact that if $A$ is a finite abelian group, then the homomorphisms $\chi:A \to \CC^*$ form a basis for the space of functions $\CC^A$. Hence, if $f:A\to \CC$ is a function such that $\sum_{a\in A} \chi(a)f(a) = 0$ for all characters $\chi$, then $f = 0$. 

\begin{proof}[Proof of Corollary \ref{BA36}]
We claim that under the assumptions of the Corollary,
\[
\psi\big(U(F_{q^m}) \big) = U'(F_{q^m})
\]
for all $n \ge 1$. By change of base from $F_q$ to $F_{q^m}$, the claim reduces to the case $m=1$.  By \ref{BA34}(*) we have 
\begin{align*}
\sum_{E \in J_\mM(F_q)}
&\chi(E) \sum_{r=0}^\infty
\sharp 
\Big\{ D \in \Div_{\ge 0}(U) 
\; \big| \; 
D \equiv_\mM E +rc 
\Big\} T^r
\\
&= 
L(T, C, \chi)
\\
&=
L(T, C'_{F_{q}}, \chi' )
\\
&=
\sum_{E' \in J_{\mM'}(F_q)}
\chi'(E') \sum_{r=0}^\infty
\sharp 
\Big\{ D' \in \Div_{\ge 0}(U') 
\; \big| \; 
D' \equiv_{\mM'} E' +rc' 
\Big\} T^r
\\
&=
\sum_{E \in J_{\mM}(F_q)}
\chi(E) \sum_{r=0}^\infty
\sharp 
\Big\{ D' \in \Div_{\ge 0}(U') 
\; \big| \; 
D' \equiv_{\mM'} \psi(E) +rc' 
\Big\} T^r
\end{align*}
for all characters $\chi$ of $J_\mM(F_q)$. Since characters form a basis for the space of functions (\ref{BB34}), it follows that 
\[
\sharp 
\Big\{ D \in \Div_{\ge 0}(U) 
\; \big| \; 
D \equiv_\mM E +rc 
\Big\}
=
\sharp 
\Big\{ D' \in \Div_{\ge 0}(U') 
\; \big| \; 
D' \equiv_{\mM'} \psi(E) +rc' 
\Big\} 
\]
for all $r$ and all $E \in J_\mM(F_q)$. According to Lemma \ref{BB33} applied to the case $r=1$ it follows that $E \in U(F_q)$ if and only $\psi(E) \in U'(F_q)$, which establishes the claim. The Corollary now follows from Theorem \ref{BC33}.
\end{proof}

\bibliographystyle{plain}
\bibliography{harvard}
\end{document}